\documentclass[12pt,twoside]{amsart}
\usepackage{amssymb}
\usepackage{amscd}
\usepackage{hyperref}

\title[Minimal model program]
{Minimal model program for excellent surfaces} 
\author{Hiromu Tanaka} 
\address{Department of Mathematics, Imperial College, London, 180 Queen's Gate, 
London SW7 2AZ, UK} 
\email{h.tanaka@imperial.ac.uk}
\newcommand{\Nklt}[0]{{\operatorname{Nklt}}}

\newcommand{\red}[0]{{\operatorname{red}}}

\newcommand{\Spec}[0]{{\operatorname{Spec}}}

\newcommand{\Supp}[0]{{\operatorname{Supp}}}
\newcommand{\Pic}[0]{{\operatorname{Pic}}}

\newcommand{\Ex}[0]{{\operatorname{Ex}}}
\newtheorem{thm}{Theorem}[section]
\newtheorem{lem}[thm]{Lemma}
\newtheorem{cor}[thm]{Corollary}
\newtheorem{prop}[thm]{Proposition}

\newtheorem*{claim}{Claim}

\theoremstyle{definition}

\newtheorem{dfn}[thm]{Definition}

\newtheorem{rem}[thm]{Remark}
\newtheorem*{ack}{Acknowledgments}      
         
\newtheorem{step}[thm]{Step}

\newtheorem{assumption}[thm]{Assumption}

\makeatletter
  
  \@addtoreset{equation}{section}
  \makeatother

\newcommand{\MO}{\mathcal{O}}
\newcommand{\R}{\mathbb{R}}
\newcommand{\Q}{\mathbb{Q}}
\newcommand{\Z}{\mathbb{Z}}

\newcommand{\p}{\mathfrak{p}}

\begin{document}

\maketitle

\begin{abstract}
We establish the minimal model program  
for log canonical and Q-factorial surfaces over excellent base schemes. 
\end{abstract}

\tableofcontents

\section{Introduction}

The Italian school of algebraic geometry in the early 20th century 
established the classification theory for smooth projective surfaces over the field of complex numbers, 
which was generalised by Kodaira, Shafarevich and Bombieri--Mumford. 
In particular, Shafarevich studied 
the minimal model theory for regular surfaces 
in the classical sense (\cite{Sha66}). 
Koll\'ar and Kov\'acs proved, in their unpublished preprint \cite{KK}, 
that the minimal model program holds 
for log canonical surfaces over algebraically closed fields 
by the viewpoint of the higher dimensional minimal model theory. 
Recently, Fujino and the author established 
the minimal model program for $\Q$-factorial surfaces over algebraically closed fields (\cite{Fuj11}, \cite{Tan14}).

The purpose of this paper is 
extending the minimal model program 
for surfaces over algebraically closed fields 
to the one for surfaces over excellent base schemes. 
More precisely, we prove the following theorem. 

\begin{thm}[Theorem~\ref{t-qfac-mmp}, Theorem~\ref{t-lc-mmp}]\label{intro-MMP}
Let $B$ be a regular excellent separated 
scheme of finite dimension. 
Let $\pi:X \to S$ be a projective $B$-morphism 
from a quasi-projective intergal normal $B$-scheme $X$ of dimension two 
to a quasi-projective $B$-scheme $S$. 
Let $\Delta$ be an effective $\R$-divisor on $X$. 
Assume that one of the following properties holds. 
\begin{enumerate}
\item[(LC)] $(X, \Delta)$ is log canonical.
\item[(QF)] $X$ is $\Q$-factorial and, 
for the irreducible decomposition $\Delta=\sum_{i \in I} \delta_i D_i$, 
the inequality $0 \leq \delta_i \leq 1$ holds for any $i \in I$. 
\end{enumerate}
Then, there exist a sequence of projective birational $S$-morphisms 
$$X=:X_0 \overset{\varphi_0}\to 
X_1 \overset{\varphi_1}\to \cdots 
\overset{\varphi_{r-1}}\to X_r=:X^{\dagger}$$
and $\R$-divisors $\Delta_i$ and $\Delta^{\dagger}$ defined by 
$$\Delta=:\Delta_0,\quad (\varphi_i)_*(\Delta_i)=:\Delta_{i+1}, \quad \Delta_r=:\Delta^{\dagger}$$
which satisfy the following properties. 
\begin{enumerate}
\item Each $X_i$ is a projective intergal normal $S$-scheme of dimension two.
\item Each $(X_i, \Delta_i)$ satisfies {\em (LC)} or {\em (QF)} 
according as the above assumption.
\item $\rho(X_{i+1}/S)=\rho(X_i/S)-1$ for any $i$. 
\item $(X^{\dagger}, \Delta^{\dagger})$ satisfies one of the following conditions.
\begin{enumerate}
\item $K_{X^{\dagger}/B}+\Delta^{\dagger}$ is nef over $S$.
\item There is a projective $S$-morphism $\mu:X^{\dagger} \to Z$ 
to a projective $S$-scheme $Z$ 
such that $\mu_*\mathcal O_{X^{\dagger}}=\mathcal O_Z$, 
$\dim X^{\dagger}>\dim Z$, $-(K_{X^{\dagger}/B}+\Delta^{\dagger})$ is $\mu$-ample 
and $\rho(X^{\dagger}/S)-1=\rho(Z/S)$.
\end{enumerate}
\end{enumerate}
\end{thm}

If $({\rm a})$ holds, then $(X^{\dagger}, \Delta^{\dagger})$ is called 
a log minimal model over $S$. 
If $({\rm b})$ holds, then $\mu:X^{\dagger} \to Z$ is called 
a $(K_{X^{\dagger}/B}+\Delta^{\dagger})$-Mori fibre space over $S$.

Many results in minimal model theory 
depend on the lower dimensional facts, 
hence Theorem~\ref{intro-MMP} might be applied to establish 
the three dimensional minimal model program over 
$\Spec\,\Z$ or imperfect fields.

Also, it is useful to treat varieties over excellent base schemes 
even when one studies the minimal model theory 
over an algebraically closed field of positive characteristic. 
For instance, a known proof for the fact that $3$-dimensional terminal singularities are isolated 
depends on the result that terminal excellent surfaces are regular 
(cf. \cite[Corollary 2.30]{Kol13}). 

As applications of Theorem~\ref{intro-MMP}, we also prove inversion of adjunction (Theorem~\ref{t-IOA}) and 
the Koll\'ar--Shokurov connectedness theorem (Theorem~\ref{t-connected}) 
for surfaces over excellent schemes. 
The former is a direct consequence of the minimal model program, 
whilst the latter needs some arguments.

\begin{rem}
It is worth explaining why 
the schemes that appear in Theorem~\ref{intro-MMP} 
are assumed to be not only Noetherian but also excellent. 
We will frequently use the desingularisation theorem 
by Lipman which holds for quasi-excellent surfaces \cite{Lip78}. 
Furthermore, in order to define canonical divisors, 
it seems to be necessary to assume that the base scheme $B$ has a dualising complex. 
In particular, $B$ needs to be universally catenary (cf. \cite[Ch. III, Section 10]{Har66}). 
Moreover, following \cite{Kol13}, we restrict ourselves to treating only regular base schemes. 
\end{rem}

\subsection{Proof of Theorem~\ref{intro-MMP}}\label{ss-sketch}

Let us overview some of the ideas of the proof of Theorem~\ref{intro-MMP}. 
The main strategy is to imitate the proof in characteristic zero. 
To this end, we establish appropriate vanishing theorems. 
After that, we prove the base point free theorem 
by applying the standard argument called the X-method. 
Some of the vanishing theorems we will establish are as follows. 

\begin{thm}[Theorem~\ref{t-rel-klt-kvv}]\label{intro-rel-klt-kvv}
Let $B$ be a regular excellent separated 
scheme of finite dimension. 
Let $\pi:X \to S$ be a projective $B$-morphism 
from a two-dimensional quasi-projective klt pair $(X, \Delta)$ over $B$ 
to a quasi-projective $B$-scheme $S$. 
Let $D$ be a $\Q$-Cartier $\Z$-divisor on $X$ 
such that $D-(K_{X/B}+\Delta)$ is $\pi$-ample. 
If $\dim \pi(X) \geq 1$, then the equation $R^i\pi_*\MO_X(D)=0$ holds for every $i>0$. 
\end{thm}

\begin{thm}[Theorem~\ref{t-glob-klt-kvv}]\label{intro-glob-klt-kvv}
Let $k$ be a field of characteristic $p>0$. 
Let $(X, \Delta)$ be a two-dimensional projective klt pair over $k$. 
Let $D$ be a $\Q$-Cartier $\Z$-divisor on $X$ 
such that $D-(K_X+\Delta)$ is ample. 
Let $N$ be a nef Cartier divisor on $X$ with $N\not\equiv 0$. 
Then there exists $m_0\in\mathbb Z_{>0}$ such that 
$$H^{i}(X, D+mN)=0$$
for any integers $i$ and $m$ satisfying $i>0$ and $m \geq m_0$. 
\end{thm}

Let us overview how to show the above vanishing theorems.  
Theorem~\ref{intro-rel-klt-kvv} is well-known for the case 
where $\dim \pi(X)=2$ (cf. \cite[Theorem 10.4]{Kol13}). 
One of the essential properties in this case is that 
the intersection matrix of the exceptional locus is negative definite. 
For the remaining case i.e. $\dim \pi(X)=1$, 
the intersection matrix of any fibre of $\pi$ is not negative definite 
but semi-negative definite. 
Hence we can apply an almost parallel argument to 
the one of \cite[Theorem 10.4]{Kol13} 
after replacing some strict inequality signs \lq\lq $>$" by \lq\lq $\geq$".

Theorem~\ref{intro-glob-klt-kvv} is known 
if $k$ is algebraically closed (\cite[Theorem 2.11]{Tan15}). 
The proof of \cite[Theorem 2.11]{Tan15} depends on 
Kawamata's covering trick, which unfortunately seems to need the assumption 
that $k$ is algebraically closed (cf. \cite[Lemma 1-1-2]{KMM87}). 
In order to overcome this technical difficulty, 
we make use of the notion of $F$-singularities. 
Actually, we first establish a similar vanishing theorem 
for strongly $F$-regular surfaces (Proposition~\ref{p-F-vanishing}) 
and second extend it to the klt case via log resolutions.

\subsection{Related results}\label{ss-related}

On the classification theory for smooth surfaces over algebraically closed fields, 
we refer to \cite{Bea83}, \cite{Bad01}. 
Theories of log surfaces and normal surfaces have been developed by Iitaka, Kawamata, Miyanishi, Sakai and many others 
(cf. \cite{Miy81}, \cite{Sak87}).

It is worth mentioning on the fact that 
Theorem~\ref{intro-MMP} is known for some special cases.  
\begin{enumerate}
\item[(i)] $B$ is an algebraically closed field. 
\item[(ii)] $X$ is regular and $\Delta=0$. 
\end{enumerate}

For the former case (i), 
Theorem~\ref{intro-MMP} is a known result (cf. \cite{Fuj11}, \cite{FT12}, \cite{Tan14}). 
Concerning the latter case (ii), 
let us review that Theorem~\ref{intro-MMP} follows from a combination of known results. 
Thanks to an analogue of Castelnuovo's contractiblity criterion 
(cf. \cite[Theorem 10.5]{Kol13}), 
we can repeatedly contract a $(-1)$-curve 
i.e. the curve $E$ appearing in \cite[Theorem 10.5]{Kol13}, 
and this process will terminate. 
Since we can make use of the cone theorem (cf. Subsection~\ref{ss-cone}), 
we deduce from the Riemann--Roch theorem that 
the program ends with either a minimal model or a Mori fibre space. 

For the general case, we would like a replacement of Castelnuovo's contractiblity criterion, 
which is nothing but the base point free theorem 
established by the vanishing theorems mentioned in Subsection~\ref{ss-sketch}.

\begin{rem}
Theorem~\ref{intro-MMP} was proved 
under the assumption that $(X, \Delta)$ is a log canonical surface 
over a field $k$ of characteristic $p>0$ by the author in an earlier draft. 
In \cite{BCZ15}, 
Birkar--Chen--Zhang independently obtained the same result 
under the assumption that $(X, \Delta)$ is a klt surface over $k$. 
Both the proofs depend on Keel's result \cite{Kee99} which differ from ours. 
\end{rem}

\begin{ack}
Part of this work was done whilst 
the author visited National Taiwan University in December 2013 
with the support of the National Center for Theoretical Sciences. 
He would like to thank Professor Jungkai Alfred Chen 
for his generous hospitality. 
The author would like to thank Professors 
Caucher Birkar, 
Paolo~Cascini, 
Yoshinori~Gongyo, 
J\'anos~Koll\'ar, 
Mircea~Musta\c{t}\u{a}, 
Yusuke Nakamura, 
Yuji Odaka, 
Shunsuke Takagi, 
Chenyang~Xu 
for very useful comments and discussions. 
The author would like to thank the referee for reading 
manuscript carefully and for suggesting several improvements.  
This work was supported by JSPS KAKENHI (no. 24224001) 
and EPSRC. 
\end{ack}

\section{Preliminaries}

\subsection{Notation}\label{ss-notation}

In this subsection, we summarise notation used in this paper. 

\begin{itemize}
\item We will freely use the notation and terminology in \cite{Har77} 
and \cite{Kol13}. 
\item 
A ring $A$ is {\em quasi-excellent} if 
$A$ is a Noetherian G-ring such that for any finitely generated $A$-algebra $B$, 
its regular locus 
$$\{\p \in \Spec\,B\,|\,B_{\p}\text{ is a regular local ring}\}$$
is an open subset of $\Spec\,B$ (cf. \cite[\S 32]{Mat89}). 
A ring $A$ is {\em excellent} if $A$ is quasi-excellent and universally catenary. 
\item We say a scheme $X$ is {\em regular} (resp. {\em Gorenstein}) if 
the local ring $\MO_{X, x}$ at any point $x \in X$ is regular (resp. Gorenstein). 
We say a scheme $X$ is {\em excellent} 
if $X$ is quasi-compact and 
any point $x \in X$ admits an affine open neighbourhood $U \simeq \Spec\,A$ 
for some excellent ring $A$. 
\item For a scheme $X$, its {\em reduced structure} $X_{\red}$ 
is the reduced closed subscheme of $X$ such that the induced morphism 
$X_{\red} \to X$ is surjective. 
\item For an integral scheme $X$, 
we define the {\em function field} $K(X)$ of $X$ 
to be $\MO_{X, \xi}$ for the generic point $\xi$ of $X$. 
\item For a scheme $S$, 
we say $X$ is a {\em variety over} $S$ or an $S$-{\em variety} if 
$X$ is an integral scheme that is separated and of finite type over $S$. 
We say $X$ is a {\em curve} over $S$ or an $S$-{\em curve} 
(resp. a {\em surface} over $S$ or an $S$-{\em surface}) 
if $X$ is an $S$-variety of dimension one (resp. two). 
\item 
Let $\Delta$ be an $\R$-divisor on a Noetherian normal scheme 
and let $\Delta=\sum_{i \in I}\delta_i\Delta_i$ be the irreducible decomposition. 
Take a real number $a$. 
We write $\Delta \leq a$ 
if $\delta_i \leq a$ for any $i \in I$. 
We set $\Delta^{\leq a}:=\sum_{i \in I, \delta_i\leq a}\delta_i\Delta_i$. 
We define $\Delta \geq a$ and $\Delta^{\geq a}$ etc, in the same way. 
\item For a field $k$, let $\overline k$ be an algebraic closure of $k$. 
If $k$ is of characteristic $p>0$, 
then we set $k^{1/p^{\infty}}:=\bigcup_{e=0}^{\infty} k^{1/p^e}
=\bigcup_{e=0}^{\infty} \{x \in \overline k\,|\, x^{p^e} \in k\}$. 
\item We say a field $k$ of positive characteristic 
is $F$-{\em finite} if $[k:k^p]<\infty$. 
\item Let $f:X \to Y$ be a projective morphism of 
Noetherian schemes, 
where $X$ is an integral normal scheme. 
For a Cartier divisor $L$ on $X$, 
we say $L$ is $f$-{\em free} 
if the induced homomorphism $f^*f_*\MO_X(L) \to \MO_X(L)$ is surjective. 
Let $M$ be an $\R$-Cartier $\R$-divisor on $X$. 
We say $M$ is $f$-{\em ample} (resp. $f$-semi-ample) 
if we can write $M=\sum_{i=1}^r a_iM_i$ 
for some $r \geq 1$, positive real numbers $a_i$ and 
$f$-ample (resp. $f$-free) Cartier divisors $M_i$. 
We say $M$ is $f$-{\em big} if we can write $M=A+E$ 
for some $f$-ample $\R$-Cartier $\R$-divisor $A$ and effective $\R$-divisor $E$. 
We can define $f$-{\em nef} $\R$-divisors in the same way as in \cite[Definition 1.4]{Kol13}. 
We say $M$ is $f$-{\em numerically-trivial}, denoted by $M \equiv_f 0$, 
if $M$ and $-M$ are $f$-nef. 
\end{itemize}

\subsubsection{Excellent schemes}\label{ss-excellent}

If $B$ is an excellent scheme, then any scheme $X$ that is of finite type over $B$ 
is again excellent (cf. \cite[\S 32]{Mat89}). 
We can freely use resolution of singularities for excellent surfaces 
(cf. \cite{Lip78}, Subsection \ref{ss-log-resolution}). 
Also excellent schemes satisfy reasonable dimension formulas (cf. \cite[Section 8.2]{Liu02}).

\subsubsection{Base schemes}

We basically work over a base scheme $B$ satisfying the following properties. 

\begin{assumption}\label{a-base}
$B$ is a scheme which is excellent, regular, separated, and of finite dimension. 
\end{assumption}

\subsubsection{Quasi-projective morphisms}\label{ss-proj-morph}

Let $f:X \to S$ be a morphism of Noetherian schemes. 
We say $f$ is {\em projective} (resp. {\em quasi-projective}) 
if there exists a closed immersion (resp. an immersion) 
$X \hookrightarrow \mathbb P^n_S$ over $S$ for some non-negative integer $n$. 

\begin{rem}\label{r-projective}
We adopt the definition of projective morphisms by \cite{Har77}. 
If there is an ample invertible sheaf on $S$ 
in the sense of \cite[page 153]{Har77}, 
then our definition coincides with the one of Grothendieck 
(cf. \cite[Section 5.5.1]{FGAex}).  
\end{rem}

\subsubsection{Perturbation of nef and big divisors}

We here give two remarks on fundamental properties of divisors, which is well-known for varieties over fields. 

\begin{rem}\label{r-fibrewise-ample}
Let $f:X \to Y$ be a projective morphism of Noetherian schemes, 
where $X$ is an integral normal scheme. 
Let $L$ be an $\R$-Cartier $\R$-divisor. 
Then the following are equivalent. 
\begin{enumerate}
\item 
$L$ is $f$-ample. 
\item 
For any closed point $y \in Y$ and any integral closed subscheme $V$ of $X_y$, 
it holds that $V \cdot L^{\dim V}>0$. 
\end{enumerate}
Indeed, clearly (1) implies (2), hence it suffices to prove the opposite implication. 
If $L$ is a $\Q$-Cartier $\Q$-divisor, 
then it is well-known that (2) implies (1) (cf. \cite[Proposition 1.41]{KM98}). 
Let us consider the general case. 
Assume that $L$ satisfies (2). 
We can write $L=\sum_{i=1}^n r_iL_i$ 
for real numbers $r_1, \cdots, r_n$ and Cartier divisors $L_1, \cdots, L_n$. 
Thanks to the Nakai's criterion for projective schemes over fields (cf. \cite[Theorem 2.3.18]{Laz04a}) 
and openness of amplitude  (cf. \cite[Proposition 1.41]{KM98}), 
if $q_i$ is a rational number sufficiently close to $r_i$ for any $i \in\{1, \cdots, n\}$, 
then also $L':=\sum_{i=1}^n q_iL_i$ satisfies (2). 
Thus, we can write $L=\sum_{j=1}^m s_j L'_j$ 
such that $s_j \in \R_{>0}$, $\sum_{j=1}^m s_j=1$ and $L'_j$ is an ample $\Q$-Cartier $\Q$-divisor. Hence, $L$ satisfies (1). 
\end{rem}

\begin{rem}\label{r-Kodaira-lemma}
Let $f:X \to Y$ be a projective morphism of Noetherian schemes, 
where $X$ is an integral normal scheme. 
Let $M$ be an $f$-nef and $f$-big $\R$-Cartier $\R$-divisor on $X$. 
Since $M$ is $f$-big, we can write $M=A+E$ 
for some $f$-ample $\R$-divisor $A$ and effective $\R$-divisor $E$. 
For any $\epsilon \in \R$ such that $0<\epsilon \leq 1$, we obtain 
$$M=A+E=(1-\epsilon)M+\epsilon(A+E)=A_{\epsilon}+\epsilon E,$$
where $A_{\epsilon}:=(1-\epsilon)M+\epsilon A$. 
Since $M$ is $f$-nef and $A$ is $f$-ample, 
Remark \ref{r-fibrewise-ample} implies that $A_{\epsilon}$ is $f$-ample. 
Therefore, $M-\epsilon E$ is $f$-ample for any real number $\epsilon$ such that 
$0<\epsilon \leq 1$. 
\end{rem}

\subsubsection{Canonical sheaves}\label{ss-cano-div}

Let $B$ be a scheme satisfying Assumption~\ref{a-base}. 
Let $X$ be a quasi-projective normal $B$-variety, equipped with the structure morphism $\alpha:X \to B$. 
For the definition of the {\em canonical sheaf} $\omega_{X/B}$, 
we refer to \cite[Definition 1.6]{Kol13}. 
Note that $\omega_{X/B}$ coincides 
with $\mathcal H^r(\alpha^!\MO_B)$, 
where $\alpha^!$ is defined as in \cite[Ch. III, Theorem 8.7]{Har66} 
and 
$r$ is the integer 
such that $\mathcal H^r(\alpha^!\MO_B)|_{X^{\text{reg}}} \neq 0$ for 
the regular locus $X^{\text{reg}}$ of $X$. 
There is a $\Z$-divisor $K_{X/B}$, called the {\em canonical divisor}, such that $\omega_{X/B} \simeq \MO_X(K_{X/B})$. 
Note that $K_{X/B}$ is uniquely determined up to linear equivalences. 

\begin{lem}\label{l-bc}
Let $\alpha:X \xrightarrow{\beta} B' \xrightarrow{\theta} B$ 
be quasi-projective morphisms of Noetherian schemes such that 
\begin{enumerate}
\item $B$ and $B'$ satisfy Assumption~\ref{a-base}, and 
\item $X$ and $B'$ are normal $B$-varieties. 
\end{enumerate}
Then $\omega_{B'/B}$ is an invertible sheaf and 
there is an isomorphism of $\MO_X$-modules: 
$\omega_{X/B} \simeq \omega_{X/B'} \otimes_{\MO_X} \beta^*\omega_{B'/B}$. 
\end{lem}

\begin{proof}
Since $B'$ is regular and in particular Gorenstein, 
it follows that $\omega_{B'/B}$ is an invertible sheaf. 
We have that 
$$\alpha^!\MO_B \simeq \beta^!\theta^!\MO_B 
\simeq \beta^!(\omega_{B'/B}[r]) \simeq 
\beta^!\MO_{B'} \otimes \beta^*\omega_{B'/B}[r]$$
for some integer $r$, 
where the first isomorphism holds by \cite[Ch. III, 2) of Proposition 8.7]{Har66}, 
the second one follows from \cite[Ch. V, Example 2.2 and Theorem 3.1]{Har66}
and the third one holds by  
\cite[Ch. III, 6) of Proposition 8.8]{Har66}. 
Taking the $s$-th cohomology sheaf $\mathcal H^s$ of the both hand sides for an appropriate integer $s$, 
we have that 
$$\omega_{X/B} \simeq \omega_{X/B'} \otimes \beta^*\omega_{B'/B},$$
as desired.  
\end{proof}

\subsubsection{Singularities of minimal model program}\label{ss-mmp-sing}

Let $B$ be a scheme satisfying Assumption~\ref{a-base}. 
Let $X$ be a quasi-projective normal $B$-variety. 
We say $(X, \Delta)$ is a {\em log pair} over $B$ 
if $\Delta$ is an effective $\R$-divisor on $X$ such that 
$K_{X/B}+\Delta$ is $\R$-Cartier. 
Although the book \cite[Definition 2.8]{Kol13} gives the definition 
of terminal, canonical, klt, plt, dlt, log canonical singularities 
only for the case where $\Delta$ is a $\Q$-divisor, 
we can extend them for the case where $\Delta$ is an $\R$-divisor in the same way. 
Note that we always assume that $\Delta$ is effective 
although \cite[Definition 2.8]{Kol13} does not impose this assumption.

\subsubsection{Log resolutions}\label{ss-log-resolution}

Let $B$ be a scheme satisfying Assumption~\ref{a-base}. 
Let $X$ be a quasi-projective surface over $B$. 
For a closed subset $Z$ of $X$, 
we say $f:Y \to X$ is a {\em log resolution} of $(X, Z)$ 
if $f$ is a projective birational $B$-morphism from a quasi-projective $B$-surface 
such that $f^{-1}(Z) \cup \Ex(f)$ is a simple normal crossing divisor, 
where we refer to \cite[Definition 1.7]{Kol13} 
for the definition of simple normal crossing divisors. 
By \cite{Lip78}, 
there exists a log resolution of $(X, Z)$ 
for any quasi-projective $B$-surface $X$ and any closed subset $Z$ of $X$. 
For an $\R$-divisor $\Delta$, a log resolution of $(X, \Delta)$ means 
a log resolution of $(X, \Supp \Delta)$.

\subsubsection{Multiplier ideals}\label{ss-multiplier}

Let $B$ be a scheme satisfying Assumption~\ref{a-base}. 
Let $X$ be a quasi-projective normal surface over $B$. 
Let $\Delta$ be an $\R$-divisor such that $K_{X/B}+\Delta$ is $\R$-Cartier. 
For a log resolution $\mu:V \to X$ of $(X, \Delta)$, 
we define the {\em multiplier ideal} of $(X, \Delta)$ by 
$$\mathcal J_{\Delta}:=\mu_*\MO_V(\ulcorner K_{V/B}-\mu^*(K_{X/B}+\Delta)\urcorner).$$
By the same argument as in \cite[Theorem~9.2.18]{Laz04}, 
we can show that the coherent sheaf $\mathcal J_{\Delta}$ 
does not depend on the choice of log resolutions. 
If $\Delta$ is effective, then it follows that $\mathcal J_{\Delta} \subset \MO_X$. 

\subsection{Intersection numbers and Riemann--Roch theorem}

We recall the definition of intersection numbers (cf. \cite[Ch I, Sections~1 and 2]{Kle66}).

\begin{dfn}\label{def-intersection}
\begin{enumerate}
\item{Let $C$ be a projective curve over a field $k$. 
Let $M$ be an invertible sheaf on $C$. 
It is well-known that 
$$\chi(C, mM)=\dim_k(H^0(C, mM))-\dim_k(H^1(C, mM)) \in\Z[m]$$ 
and the degree of this polynomial is at most one 
(cf. \cite[Ch I, Section 1, Theorem in page 295]{Kle66}). 
We define the {\em degree} of $M$ over $k$, 
denoted by $\deg_k M$ or $\deg M$, to be the coefficient of $m$. }
\item{Let $B$ be a scheme satisfying Assumption~\ref{a-base}. 
Let $X$ be a quasi-projective $B$-scheme, and 
let $C \hookrightarrow X$ be a closed immersion 
such that $C$ is a projective $B$-curve 
and the structure morphism $C \to B$ factors through $\Spec\,k$ 
for a field $k$. 
In particular, $C$ is projective over $k$ (cf. \cite[Ch II, Theorem 7.6]{Har77}). 
Let $L$ be an invertible sheaf on $X$. 
We define the {\em intersection number} over $k$, 
denoted by $L\cdot_k C$ or $L \cdot C$, to be $\deg_{k}(L|_C).$}
\end{enumerate}
\end{dfn}

The following is the Riemann--Roch theorem for curves. 

\begin{thm}\label{RR-curve}
Let $k$ be a field. 
Let $X$ be a projective curve over $k$. 
\begin{enumerate}
\item{If $L$ is a Cartier divisor on $X$, then 
$$\chi(X, L)=\deg_k(L)+\chi(X, \MO_X).$$}
\item{Assume that $X$ is regular. 
Let $D$ be a $\Z$-divisor and 
let $D=\sum_{i\in I} a_iP_i$ be the prime decomposition. 
Then 
$$\deg_k(D)=\sum_{i\in I} a_i\dim_k(k(P_i))$$
where $k(P_i)$ is the residue field at $P_i$. }
\end{enumerate}
\end{thm}

\begin{proof}
The assertion (1) follows from Definition~\ref{def-intersection}(1). 
The assertion (2) holds by the same argument as the case where $k$ is algebraically closed 
(cf. the proof of \cite[Ch~IV, Theorem~1.3]{Har77}). 
\end{proof}

As a corollary, 
we obtain a formula between $\deg(\omega_X)$ and $\chi(X, \MO_X)$ for Gorenstein curves.

\begin{cor}\label{genus-formula}
Let $k$ be a field. 
Let $X$ be a projective Gorenstein curve over $k$. 
Then 
$$\deg_k(\omega_X)=2(h^1(X, \MO_X)-h^0(X, \MO_X))=-2\chi(X, \MO_X).$$
\end{cor}

\begin{proof}
By Theorem~\ref{RR-curve}(1), we obtain $\chi(X, \omega_X)=\deg(\omega_X)+\chi(X, \MO_X).$ 
By Serre duality, we obtain $\chi(X, \omega_X)=-\chi(X, \MO_X)$, which implies the assertion. 
\end{proof}

\begin{cor}\label{adjunction-genus}
Let $B$ be a scheme satisfying Assumption~\ref{a-base}. 
Let $X$ be a quasi-projective regular $B$-surface, and 
let $C \hookrightarrow X$ be a closed immersion over $B$ 
such that $C$ is a projective $B$-curve 
and the structure morphism $C \to B$ factors through $\Spec\,k$ 
for a field $k$. 
Then, 
$$(K_{X/B}+C)\cdot_k C=\deg_k \omega_{C/B}=\deg_k \omega_{C/k}$$
$$=2(\dim_kH^1(C, \MO_C)-\dim_kH^0(C, \MO_C)).$$
\end{cor}

\begin{proof}
The first equality holds by the adjunction formula (cf. \cite[(4.1.3)]{Kol13}), 
the second one by Lemma~\ref{l-bc}, and the third one 
by Corollary~\ref{genus-formula}. 
\end{proof}

The following is the Riemann--Roch theorem for surfaces.

\begin{thm}\label{RR-surface}
Let $k$ be a field. 
Let $X$ be a projective regular surface over $k$. 
Let $D$ be a $\Z$-divisor on $X$. 
Then 
$$\chi(X, D)=\chi(X, \MO_X)+\frac{1}{2}D\cdot_k (D-K_X).$$
\end{thm}

\begin{proof}
We can write $D=\sum_i A_i-\sum_j B_j$, where all $A_i$ and $B_j$ are prime divisors. 
Thus, it suffices to prove that if $D$ satisfies the required equation, so do $D+C$ and $D-C$ 
for a prime divisor $C$ on $X$. 
These can be checked directly by using Corollary~\ref{adjunction-genus}. 
\end{proof}

\subsection{Negativity lemma}\label{ss-negativity}

\begin{lem}\label{l-negativity}
Let $B$ be a scheme satisfying Assumption~\ref{a-base}. 
Let $f:X \to Y$ be a projective birational $B$-morphism of 
quasi-projective normal $B$-surfaces. 
Let $-B$ be an $f$-nef $\R$-Cartier $\R$-divisor on $X$. 
Then the following hold. 
\begin{enumerate}
\item $B$ is effective if and only if $f_*B$ is effective. 
\item Assume that $B$ is effective. 
For any point $y \in Y$, one of $f^{-1}(y) \subset \Supp B$ and 
$f^{-1}(y) \cap \Supp B=\emptyset$ holds. 
\end{enumerate}
\end{lem}

\begin{proof}
We can apply the same argument as in \cite[Lemma 3.39]{KM98} 
by using \cite[Theorem 10.1]{Kol13}. 
\end{proof}

\subsection{Cone theorem}\label{ss-cone}

Let $B$ be a scheme satisfying Assumption~\ref{a-base}. 
Let $\pi:X \to S$ be a projective $B$-morphism 
from a quasi-projective normal $B$-variety $X$ 
to a quasi-projective $B$-scheme $S$. 
Let 
$$Z(X/S)_{\R}:=\bigoplus_C \R C,$$ 
where $C$ runs over the projective $B$-curves in $X$ such that 
$\pi(C)$ is one point. 
An element $\sum \alpha_C C$ of $Z(X/S)_{\R}$ is called an 
{\em effective} $1$-{\em cycle} if each $\alpha_C$ is not negative. 
Let $N(X/S)_{\R}:=Z(X/S)_{\R}/\equiv$ be  
the quotient $\R$-vector space by numerical equivalence, 
i.e. for $\Gamma_1, \Gamma_2 \in Z(X/S)_{\R}$, 
the numerical equivalence $\Gamma_1 \equiv \Gamma_2$ holds if and only if 
$L \cdot \Gamma_1=L \cdot \Gamma_2$ 
for any invertible sheaf $L$ on $X$. 
For an element $\Gamma \in Z(X/S)_{\R}$, 
its numerical equivalence class is denoted by $[\Gamma]$. 
Let ${\rm NE}(X/S)$ be the subset of $N(X/S)_{\R}$ 
consisting of the numerical equivalence classes of effective 1-cycles. 
If $N(X/S)_{\R}$ is a finite dimensional $\R$-vector space, 
then we set $\overline{{\rm NE}}(X/S)$ 
to be the closure of ${\rm NE}(X/S)$ with respect to the Euclidean topology. 
In this case, we set 
$$\overline{{\rm NE}}(X/S)_{L \geq 0}:=\{[\Gamma] \in \overline{{\rm NE}}(X/S)\,|\,L \cdot \Gamma \geq 0\}$$
for any $\R$-Cartier $\R$-divisor $L$ on $X$. 

\begin{rem}
For the Stein factorisation $\pi:X \to T \to S$ of $\pi$, 
we have that $N(X/S)_{\R}=N(X/T)_{\R}$ and ${\rm NE}(X/S)={\rm NE}(X/T)$. 
\end{rem}

\begin{lem}\label{l-pic-number}
Let $B$ be a scheme satisfying Assumption~\ref{a-base}. 
Let $\pi:X \to S$ be a projective $B$-morphism 
from a quasi-projective normal $B$-surface to a quasi-projective $B$-scheme. 
Then the following hold. 
\begin{enumerate}
\item If $\dim \pi(X) \geq 1$, then 
there exist finitely many projective $B$-curves $C_1, \cdots, C_q$ 
such that each $\pi(C_i)$ is one point and 
$${\rm NE}(X/S)=\sum_{i=1}^q \R_{\geq 0}[C_i].$$
\item $N(X/S)_{\R}$ is a finite dimensional $\R$-vector space. 
\end{enumerate}
\end{lem}

\begin{proof}
Replacing $\pi$ by its Stein factorisation, we may assume that $\pi_*\MO_X=\MO_S$. 

We show (1). 
We only treat the case where $\dim S=1$ because the proofs are the same. 
Since the generic fibre $X \times_S K(S)$ is geometrically irreducible (cf. \cite[Lemma 2.2]{Tan16}), 
there is a non-empty open subset $S^0$ of $S$ such that 
$X_s$ is an irreducible curve over $k(s)$ for any $s \in S^0$. 
Since $S \setminus S^0$ is a finite set, 
$\text{NE}(X/S)$ is generated by only finitely many $B$-curves 
$C_1, \cdots, C_q$. 
Thus (1) holds. 

We show (2). 
If $\dim S \geq 1$, then the assertion holds by (1). 
If $\dim S=0$, then the assertion follows from \cite[Ch. IV, Section 1, Prop 4]{Kle66}. 
In any case, (2) holds. 
\end{proof}

\begin{thm}\label{t-cone}
Let $B$ be a scheme satisfying Assumption~\ref{a-base}. 
Let $\pi:X \to S$ be a projective $B$-morphism 
from a quasi-projective normal $B$-surface to a quasi-projective $B$-scheme. 
Let $A$ be a $\pi$-ample $\R$-Cartier $\R$-divisor on $X$ and 
let $\Delta$ be an effective $\R$-divisor on $X$ such that $K_{X/B}+\Delta$ 
is $\R$-Cartier. 
Then there exist finitely many projective $B$-curves $C_1, \cdots, C_r$ on $X$ 
that satisfy the following properties: 
\begin{enumerate}
\item $\pi(C_i)$ is one point for any $1\leq i \leq r$, and  
\item 
$\overline{{\rm NE}}(X/S)=\overline{{\rm NE}}(X/S)_{K_{X/B}+\Delta+A \geq 0}+
\sum_{i=1}^r \R_{\geq 0}[C_i]$.
\end{enumerate}
\end{thm}

\begin{proof}
Replacing $\pi$ by its Stein factorisation, we may assume that $\pi_*\MO_X=\MO_S$. 
If $\dim S \geq 1$, then the assertion holds by Lemma~\ref{l-pic-number}. 
Thus we may assume that $\dim S=0$, 
i.e. $S=\Spec\,k$ where $k$ is a field. 
If $k$ is of characteristic $p>0$, then the assertion follows 
from \cite[Theorem 7.5]{Tan16}. 
If $k$ is an algebraically closed field of characteristic zero, 
then the claim holds by \cite[Theorem 1.1]{Fuj11}. 
If $k$ is a field of characteristic zero (not necessarily algebraically closed), 
then we can apply the same argument as in \cite[Theorem 7.5]{Tan16} 
using \cite[Theorem 1.1]{Fuj11}. 
We are done. 
\end{proof}

\begin{rem}
We use the same notation as in the statement of Theorem \ref{t-cone}. 
Assume that $(X, \Delta)$ is log canonical. 
If $B$ is the affine spectrum of an algebraically closed field, then $C_i$ can be chosen 
to be a curve such that $-(K_X+\Delta) \cdot C_i \leq 3$ 
(cf. \cite[Proposition 3.8]{Fuj12}, \cite[Proposition 3.15]{Tan14}). 
The author does not know whether a similar result can be formulated 
in this generality (cf. \cite[Example 7.3]{Tan16}). 
\end{rem}

\section{Vanishing theorems of Kawamata--Viehweg type}\label{s-vanishing}

\subsection{Relative case}

A key result in this subsection is Proposition~\ref{p-rel-reg-kvv}. 
As explained in Subsection \ref{ss-sketch}, 
we can apply the same argument as in \cite[Theorem 10.4]{Kol13}. 
More specifically, our proofs of 
Lemma~\ref{l-semi-negative} and Proposition~\ref{p-rel-reg-kvv} are 
quite similar to the ones of \cite[Lemma 10.3.3]{Kol13} and \cite[Theorem 10.4]{Kol13} respectively. 
However we give proofs of them for the sake of completeness.

\begin{lem}\label{l-semi-negative}
Let $B$ be a scheme satisfying Assumption~\ref{a-base}. 
Let $\pi:X \to S$ be a projective $B$-morphism 
from a quasi-projective regular $B$-surface to a quasi-projective $B$-curve 
with $\pi_*\MO_X=\MO_S$. 
Let $s\in S$ be a closed point and 
let $\pi^*(s)=\sum_{1\leq i\leq n} c_iC_i$ be the irreducible decomposition with $c_i \in \Z_{>0}$. 
Let $D=\sum_{1\leq i\leq n} d_iC_i$ be an $\mathbb{R}$-divisor on $X$. 
If $D\not\leq 0$, then there exists an integer $j$ 
such that $1\leq j\leq n$, $d_j>0$ and $C_j \cdot D \leq 0$. 
\end{lem}

\begin{proof}
By $D\not\leq 0$, 
there exists the maximum real number $a \in \mathbb{R}_{>0}$ such that 
$c_i-ad_i \ge 0$ holds for any $i \in \{1, \cdots, n\}$. 
Then we have that $\sum_{1\leq i \leq n} (c_i-ad_i)C_i \geq 0$ and 
$c_j-ad_j=0$ for some $j\in \{1, \cdots, n\}$ with $d_j>0$. 
Therefore, we obtain 
\begin{align*}
-aC_j\cdot D
= C_j \cdot (\pi^*(s)-aD)= C_j \cdot \sum_{i\neq j}(c_i-ad_i)C_i\geq 0,
\end{align*}
which implies $C_j\cdot D \leq 0$. 
\end{proof}

\begin{prop}\label{p-rel-reg-kvv}
Let $B$ be a scheme satisfying Assumption~\ref{a-base}. 
Let $\pi:X \to S$ be a projective $B$-morphism 
from a quasi-projective regular $B$-surface $X$ 
to a quasi-projective $B$-scheme $S$. 
Let $\Delta$ be an $\mathbb{R}$-divisor on $X$ with $0\leq \Delta <1$ and 
let $L$ be a Cartier divisor 
such that $L-(K_{X/B}+\Delta)$ is $\pi$-nef and $\pi$-big. 
If $\dim \pi(X) \geq 1$, then the equation $R^i\pi_*\MO_X(L)=0$ holds for every $i>0$. 
\end{prop}

\begin{proof}
Set $K_X:=K_{X/B}$. 
By perturbing $\Delta$, we may assume $\Delta$ is a $\Q$-divisor 
such that $L-(K_X+\Delta)$ is $\pi$-ample (cf. Remark \ref{r-Kodaira-lemma}). 
We set $A:=L-(K_X+\Delta)$.

By taking the Stein factorisation of $\pi$, we may assume that $\pi_*\MO_X=\MO_S$. 
If $\dim S=2$, then the assertion follows from \cite[Theorem 10.4]{Kol13}. 
Thus we may assume that $\dim S=1$. 
Then the assertion is clear for $i > 1$. 
We prove $R^1 \pi_*\MO_X(L)=0$. 

Fix a closed point $s \in S$ and 
let $C_1, \ldots, C_n$ be the prime divisors on $X$ contracting to $s$, that is, 
$\Supp(\pi^{-1}(s))=\Supp(C_1 \cup \cdots \cup C_n)$. 
Each $C_i$ is a projective curve over $k(s)$, 
hence the intersection number over $k(s)$
of an $\R$-divisor on $X$ and $C_i$ is defined in 
Definition~\ref{def-intersection}. 
We set $h^i(C_i, F):=\dim_{k(s)}H^i(C_i, F)$ for any coherent sheaf $F$ on $C_i$. 
Note that $C_i^2 \leq 0$ for any $i \in \{1, \cdots, n\}$. 
We consider the following:

\begin{claim}
The equation $H^1(Z, L|_Z)=0$ holds for any $r_i \in \mathbb Z_{\geq 0}$ and $Z:=\sum_{i=1}^n r_iC_i$ with $\sum_{i=1}^n r_i \geq 1$. 
\end{claim}

If this claim holds, then the assertion in the proposition follows from 
the theorem on formal functions (\cite[Ch III, Theorem 11.1]{Har77}). 
We prove this claim by the induction on $\sum_{i=1}^n r_i$. 

If $\sum_{i=1}^n r_i=1$, that is, $Z=C_i$ holds for some $i$, then 
\[
h^1(C_i, L|_{C_i})=h^0(C_i, \omega_{C_i}\otimes L^{-1}|_{C_i})=0,
\]
where the second equation holds by the following calculation
\begin{align*}
\deg_{k(s)}(\omega_{C_i}\otimes L^{-1}|_{C_i})
&=(K_{X}+C_i-L)\cdot_{k(s)} C_i \\
&=(K_{X}+C_i-(K_X+\Delta+A))\cdot_{k(s)} C_i \\
&=((C_i-\Delta)-A)\cdot_{k(s)} C_i \\
&< 0.
\end{align*}
Thus we are done for the case where $\sum_{i=1}^n r_i=1$. 

We assume $\sum_{i=1}^n r_i>1$. 
For every $i$ with $r_i \geq 1$, 
we obtain the following short exact sequence:
\[
0 \to \MO_{C_i}\otimes \MO_X \big( -(Z-C_i) \big) \to \MO_Z \to \MO_{Z-C_i} \to 0.
\]
Thus, by the induction hypothesis, 
it is sufficient to find an index $j$ satisfying $r_j \geq 1$ and 
\[
h^1 \big( C_j, \MO_X(L-(Z-C_j))|_{C_j} \big) =0.
\]
Note that 
\[
h^1 \big( C_j, \MO_X(L-(Z-C_j))|_{C_j} \big) = h^0 \big( C_j, \omega_{C_j}\otimes \MO_X(-L+Z-C_j)|_{C_j} \big). 
\]
For every $i$, we see 
\begin{eqnarray*}
&&\deg_{k(s)} \big( \omega_{C_i}\otimes \MO_X(-L+Z-C_i)|_{C_i} \big)\\
&=&(K_{X}+C_i-L+Z-C_i)\cdot_{k(s)} C_i \\
&=&(K_{X}-(K_X+\Delta+A)+Z)\cdot_{k(s)} C_i \\
&<&(Z-\Delta)\cdot_{k(s)} C_i.
\end{eqnarray*}
Since $Z-\Delta \not\leq 0$, we can apply Lemma \ref{l-semi-negative} 
to $D = Z-\Delta$. 
Then, we can find an index $j$ satisfying $r_j \ge 1$ and 
\[
(Z-\Delta)\cdot C_j \leq 0. 
\]
This implies $h^1 \big( C_j, \MO_X(L-(Z-C_j))|_{C_j} \big) =0$, as desired. 
\end{proof}

\begin{thm}\label{t-rel-klt-kvv}
Let $B$ be a scheme satisfying Assumption~\ref{a-base}. 
Let $\pi:X \to S$ be a projective $B$-morphism 
from a two-dimensional quasi-projective klt pair $(X, \Delta)$ over $B$ 
to a quasi-projective $B$-scheme. 
Let $D$ be a $\Q$-Cartier $\Z$-divisor on $X$ 
such that $D-(K_{X/B}+\Delta)$ is $\pi$-nef and $\pi$-big. 
Assume that one of the following conditions holds.  
\begin{enumerate}
\item $\dim \pi(X) \geq 1$.  
\item $\dim \pi(X)=0$ and for the Stein factorisation $\pi:X \to T \to S$ of $\pi$, 
$T$ is the affine spectrum of a field of characteristic zero. 
\end{enumerate}
Then the equation $R^i\pi_*\MO_X(D)=0$ holds for every $i>0$. 
\end{thm}

\begin{proof}
Replacing $\pi$ by its Stein factorisation, we may assume that $\pi_*\MO_X=\MO_S$. 
If (2) holds, then the assertion follows from \cite[Theorem 1-2-5]{KMM87} 
after taking the base change to the algebraic closure. 
If (1) holds, then we can apply the same proof as in \cite[Theorem 1-2-5]{KMM87} 
by using Proposition~\ref{p-rel-reg-kvv} instead of \cite[Theorem 1-2-3]{KMM87}. 
\end{proof}

\begin{thm}\label{t-rel-nadel}
Let $B$ be a scheme satisfying Assumption~\ref{a-base}. 
Let $\pi:X \to S$ be a projective $B$-morphism 
from a quasi-projective normal $B$-surface to a quasi-projective $B$-scheme. 
Let $\Delta$ be an $\R$-divisor on $X$ such that $K_{X/B}+\Delta$ is $\R$-Cartier 
and let $L$ be a Cartier divisor on $X$ 
such that $L-(K_{X/B}+\Delta)$ is $\pi$-nef and $\pi$-big. 
Assume that one of the following conditions holds.  
\begin{enumerate}
\item $\dim \pi(X) \geq 1$.  
\item $\dim \pi(X)=0$ and for the Stein factorisation $\pi:X \to T \to S$ of $\pi$, 
$T$ is the affine spectrum of a field of characteristic zero. 
\end{enumerate}
Then the equation 
$R^i\pi_*(\MO_X(L) \otimes_{\MO_X} \mathcal J_{\Delta})=0$ holds for every $i>0$. 
\end{thm}

\begin{proof}
We can apply the same argument as in \cite[Theorem 2.10]{Tan15} 
by using Theorem~\ref{t-rel-klt-kvv} instead of \cite[Corollary 2.7]{Tan15}. 
\end{proof}

\subsection{Global case of positive characteristic}

In this subsection, 
we establish a vanishing theorem of Kawamata--Viehweg type 
for surfaces of positive characteristic (Theorem~\ref{t-glob-klt-kvv}). 
The heart of this subsection is Proposition~\ref{p-F-vanishing}. 
For this, we need some auxiliary results: 
Lemma \ref{l-R-Fujita} and Lemma \ref{l-top}. 
Our strategy is similar to but modified from \cite[Section 2]{Tan15}.

\begin{lem}\label{l-R-Fujita}
Let $k$ be a field. 
Let $X$ be a projective normal variety over $k$. 
Let $M$ be a coherent sheaf on $X$ and 
let $A$ be an ample $\R$-Cartier $\R$-divisor on $X$. 
Then there exists a positive integer $r(M, A)$ such that 
$$H^i(X, M\otimes_{\MO_X} \MO_X(rA+N))=0$$ 
for any $i>0$, any real number $r \geq r(M, A)$ and 
any nef $\R$-Cartier $\R$-divisor $N$ such that $rA+N$ is Cartier. 
\end{lem}

\begin{proof}
The assertion follows from the same argument as in \cite[Theorem 2.2]{Tan15}. 
For the Fujita vanishing theorem in our setting, 
we refer to \cite[Theorem 1.5]{Kee03}. 
\end{proof}

\begin{lem}\label{l-top}
Let $k$ be a field. 
Let $X$ be a projective normal variety over $k$. 
Let $M$ be a coherent sheaf on $X$ and 
let $N$ be a nef $\R$-Cartier $\R$-divisor on $X$ with $N\not\equiv 0$. 
Then there exists a positive integer $r_0$ such that 
$$H^{\dim X}(X, M\otimes_{\MO_X} \MO_X(rN+N'))=0$$ 
for any real number $r \geq r_0$ and any nef $\R$-Cartier $\R$-divisor $N'$ 
such that $rN+N'$ is Cartier. 
\end{lem}

\begin{proof}
We may assume that $k$ is an infinite field by taking the base change 
to its algebraic closure if $k$ is a finite field. 
Since $X$ is projective, there exists a surjection $\bigoplus_{j=1}^{\ell} M_j \to M$ where 
each $M_j$ is an invertible sheaf. 
Thus, we may assume that $M$ is an invertible sheaf. 
Taking a sufficiently ample hyperplane section $H$ of $X$, 
where $H$ is a projective normal variety (\cite[Theorem 7']{Sei50}), 
we obtain 
{\small $$0 \to \MO_X(M+N'') \to \MO_X(H+M+N'') \to \MO_X(H+M+N'')|_H \to 0.$$}
for any nef Cartier divisor $N''$ on $X$. 
By the Fujita vanishing theorem (\cite[Theorem 1.5]{Kee03}), we see 
$H^i(X, H+M+N'')=0$ 
for any $i>0$ and nef Cartier divisor $N''$. 
Thus, it follows that  
\begin{eqnarray*}
&&H^{\dim X}(X, M+N'')\\
&\simeq& H^{\dim X-1}(H, \MO_X(H+M+N'')|_H)\\
&=&H^{\dim H}(H, \MO_X(H+M+N'')|_H)
\end{eqnarray*}
for any nef Cartier divisor $N''$ on $X$. 
Repeating this process for $N'':=rN+N'$, 
we may assume $\dim X=1$. 
Then $N$ is ample. 
The assertion follows from Lemma \ref{l-R-Fujita}. 
\end{proof}

In the following proposition, we use the notion of $F$-singularities. 
For definitions and basic properties, we refer to 
\cite[Section 2.3]{CTX15} (cf. \cite{Sch09}). 
Although \cite{CTX15} works over an algebraically closed field of positive characteristic, 
the same argument works over $F$-finite fields.

\begin{prop}\label{p-F-vanishing}
Let $k$ be an $F$-finite field of characteristic $p>0$. 
Let $X$ be a projective normal variety over $k$. 
Assume that one of the following two conditions holds. 
\begin{enumerate}
\item{$(X, \Delta)$ is sharply $F$-pure where $\Delta$ is an effective $\Q$-divisor on $X$ 
such that $(p^{e_1}-1)(K_X+\Delta)$ is Cartier for some $e_1\in\mathbb Z_{>0}$. }
\item{$(X, \Delta)$ is strongly $F$-regular, 
where $\Delta$ is an effective $\R$-divisor on $X$ 
such that $(X, \Delta)$ is $\R$-Cartier.}
\end{enumerate}
Let $L$ be a Cartier divisor on $X$ such that $L-(K_X+\Delta)$ is ample, 
and let $N$ be a nef $\R$-Cartier $\R$-divisor on $X$ with $N\not\equiv 0$. 
Then there exists a positive integer $r_0$ such that 
the equation 
$$H^{i}(X, L+rN+N')=0$$
holds for any $i\geq \dim X-1$, 
real number $r \geq r_0$ and nef $\R$-Cartier $\R$-divisor such that 
$rN+N'$ is Cartier. 
\end{prop}

\begin{proof}
By Lemma~\ref{l-top}, we may assume $i=\dim X-1$. 
Thanks to \cite[Lemma 4.1]{CTX15},  
we may assume that (1) holds. 
It follows from (1) that we get an exact sequence (cf. \cite[Definition 2.7]{CTX15}): 
$$0\to B_e \to F_*^e(\MO_X(-(p^e-1)(K_X+\Delta))) \to \MO_X \to 0,$$
where $B_e$ is a coherent sheaf depending on $e\in e_1\mathbb Z_{>0}$. 
For any nef Cartier divisor $N''$, we obtain the exact sequence: 
\begin{eqnarray*}
&&H^{\dim X-1}(X, p^e(L+N'')-(p^e-1)(K_X+\Delta)) \\
&\to& H^{\dim X-1}(X, L+N'') \\
&\to& H^{\dim X}(X, B_e\otimes_{\MO_X} \MO_X(L+N'')).
\end{eqnarray*}
By the Fujita vanishing theorem (\cite[Theorem 1.5]{Kee03}), 
we can find $e_2\in e_1\mathbb Z_{>0}$ such that 
\begin{eqnarray*}
&&H^{\dim X-1}(X, p^{e_2}(L+N'')-(p^{e_2}-1)(K_X+\Delta))\\
&=&H^{\dim X-1}(X, p^{e_2}N''+L+(p^{e_2}-1)(L-(K_X+\Delta)))\\
&=&0
\end{eqnarray*}
for every nef Cartier divisor $N''$. 
By Lemma~\ref{l-top}, 
we can find an integer $r_0>0$, depending on $e_2$, such that 
$$H^{\dim X}(X, B_{e_2}\otimes_{\MO_X} \MO_X(L+rN+N'))=0$$
for every $r \geq r_0$ and nef $\R$-Cartier $\R$-divisor such that $rN+N'$ is Cartier. 
We are done by substituting $rN+N'$ for $N''$. 
\end{proof}


\begin{thm}\label{t-glob-klt-kvv}
Let $k$ be a field of characteristic $p>0$. 
Let $(X, \Delta)$ be a two-dimensional projective klt pair over $k$. 
Let $D$ be a $\Q$-Cartier $\Z$-divisor on $X$ 
such that $D-(K_X+\Delta)$ is nef and big. 
Let $N$ and $M_1, \cdots, M_q$ be nef $\R$-Cartier $\R$-divisors on $X$ 
with $N\not\equiv 0$. 
Then there exists a positive integer $r_0$ such that 
$$H^{i}(X, L+rN+\sum_{j=1}^qs_jM_j)=0$$
for any $i>0$ and real numbers $r, s_1, \cdots, s_q$ 
such that $r \geq r_0$, $s_j \geq 0$ and $rN+\sum_{j=1}^qs_jM_j$ is Cartier. 
\end{thm}

\begin{proof}
We treat the following three cases separately. 
\begin{enumerate}
\item $k$ is $F$-finite, $X$ is regular and $\Delta$ is simple normal crossing. 
\item $X$ is regular and $\Delta$ is simple normal crossing. 
\item The case without any additional assumptions. 
\end{enumerate}

For the case (1), the assertion in the theorem follows 
from Proposition~\ref{p-F-vanishing} 
and the fact that $(X, \Delta)$ is strongly $F$-regular 
(cf. \cite[Proposition 6.18]{ST14}). 

Assume that the conditions in (2) hold. 
Take models over a finitely generated field over $\mathbb F_p$, that is, 
we can find an intermediate field $\mathbb F_p\subset k_0\subset k$ 
with $k_0$ finitely generated over $\mathbb F_p$ 
and $X_0, \Delta_0, D_0, N_0, M_{j, 0}$ over $k_0$ whose base changes and pull-backs 
are $X, \Delta, D, N, M_j$ respectively, and these satisfy the same properties 
as in the statement. 
Since $k_0$ is $F$-finite, it follows from the case (1) that 
$$\dim_k H^i(X, D+rN+\sum_{j=1}^qs_jM_j)=
\dim_{k_0} H^i(X_0, D_0+rN_0+\sum_{j=1}^qs_jM_{j, 0})=0$$
for numbers $i, r, s_j$ as in the statement. 
Thus the assertion in the theorem follows for the case (2). 

For the case (3), 
we can apply the same proof as in \cite[Theorem 1-2-5]{KMM87} 
by using the case (2) instead of \cite[Theorem 1-2-3]{KMM87}. 
\end{proof}


\begin{thm}\label{t-glob-nadel}
Let $k$ be a field of characteristic $p>0$. 
Let $X$ be a projective normal surface over $k$ and 
let $\Delta$ be an $\R$-divisor such that $K_X+\Delta$ is $\R$-Cartier. 
Let $L$ be a Cartier divisor on $X$ 
such that $L-(K_X+\Delta)$ is nef and big. 
Let $N$ and $M_1, \cdots, M_q$ be nef $\R$-Cartier $\R$-divisors on $X$ 
with $N\not\equiv 0$. 
Then there exists a positive integer $r_0$ such that 
$$H^{i}(X, \MO_X(L+rN+\sum_{j=1}^qs_jM_j) \otimes_{\MO_X} \mathcal J_{\Delta})=0$$
for any $i >0$ and real numbers $r, s_1, \cdots, s_q$ 
such that $r \geq r_0$, $s_j \geq 0$ and $rN+\sum_{j=1}^qs_jM_j$ is Cartier. 
\end{thm}

\begin{proof}
We can apply the same argument as in \cite[Theorem 2.9]{Tan15} 
by using Theorem~\ref{t-glob-klt-kvv} (resp. Theorem~\ref{t-rel-klt-kvv}) 
instead of \cite[Theorem 2.6]{Tan15} (resp. \cite[Theorem 2.7]{Tan15}). 
\end{proof}

\section{Minimal model program}

\subsection{Base point free theorem}

In Section \ref{s-vanishing}, we established some vanishing theorems of Kawamata--Viehweg type. 
By applying standard arguments called the X-method, 
we show the base point free theorem of Kawamata--Shokurov type 
(Theorem~\ref{t-bpf}). 
We start with a non-vanishing theorem of Shokurov type. 

\begin{thm}\label{t-non-vanishing}
Let $X$ be a projective regular variety over a field $k$ with $\dim X \leq 2$. 
Let $\Delta$ be an $\R$-divisor on $X$ 
such that $\llcorner \Delta \lrcorner \leq 0$ and 
$\{\Delta\}$ is simple normal crossing. 
Let $D$ be a nef Cartier divisor on $X$. 
Assume that the following conditions hold. 
\begin{enumerate}
\item $aD-(K_X+\Delta)$ is nef and big for some positive real number $a$. 
\item If $k$ is of positive characteristic and $\dim X=2$, then $D \not\equiv 0$. 
\end{enumerate}
Then there exists a positive integer $m_0$ such that 
$H^0(X, mD-\llcorner \Delta \lrcorner) \neq 0$ 
for any integer $m$ with $m \geq m_0$. 
\end{thm}

\begin{proof}
If $k$ is of characteristic zero, 
then the assertion follows from \cite[Theorem 2-1-1]{KMM87} 
by taking the base change to its algebraic closure. 
Thus we may assume that $k$ is of positive characteristic.

We first treat the case where $\dim X=1$. 
If $D \not\equiv 0$, then the assertion follows from the Serre vanishing theorem. 
Thus we may assume that $D \equiv 0$. 
We have that 
{\small 
$$h^0(X, mD-\llcorner \Delta \lrcorner)=
\chi(X, mD-\llcorner \Delta \lrcorner)=\chi(X, -\llcorner \Delta \lrcorner)
=h^0(X, -\llcorner \Delta \lrcorner) \geq 1,$$}
where the second equation follows from the Riemann--Roch theorem 
(cf. Theorem~\ref{RR-curve}) 
and the first and third equations hold by 
the Kodaira vanishing theorem for curves. 

Thus we may assume that $\dim X=2$. 
Since $D \not\equiv 0$ and $D$ is nef, 
it follows from Serre duality that 
$H^2(X, mD-\llcorner \Delta \lrcorner)=0$ for any $m \gg 0$. 
By the Riemann--Roch theorem (cf. Theorem~\ref{RR-surface}), 
we have that 
$$h^0(X, mD-\llcorner \Delta \lrcorner)
\geq 
\frac{1}{2}(mD-\llcorner \Delta \lrcorner) \cdot 
(mD-\llcorner \Delta \lrcorner-K_X)+\chi(X, \MO_X)$$
$$\geq \frac{1}{2}mD \cdot (aD-2\llcorner \Delta \lrcorner-K_X)+
\frac{1}{2}(-\llcorner \Delta \lrcorner) \cdot (-\llcorner \Delta \lrcorner-K_X)
+\chi(X, \MO_X)$$
for any $m \gg 0$. 
Thus it suffices to show that 
$D \cdot (aD-2\llcorner \Delta \lrcorner-K_X)>0$, 
which holds by 
$$D \cdot (aD-2\llcorner \Delta \lrcorner-K_X) 
\geq D \cdot (aD-(K_X+\Delta))>0.$$
We are done. 
\end{proof}


\begin{thm}\label{t-bpf}
Let $B$ be a scheme satisfying Assumption~\ref{a-base}. 
Let $\pi:X \to S$ be a projective $B$-morphism 
from a quasi-projective normal $B$-surface to a quasi-projective $B$-scheme. 
Assume that there are an effective $\R$-divisor $\Delta$ on $X$ and 
a nef Cartier divisor $D$ on $X$ which satisfy the following properties: 
\begin{enumerate}
\item $\llcorner \Delta \lrcorner=0$, 
\item $K_{X/B}+\Delta$ is $\R$-Cartier, 
\item $D-(K_{X/B}+\Delta)$ is $\pi$-nef and $\pi$-big, and
\item if $T$ is the affine spectrum of a field of positive characteristic 
for the Stein factorisation $\pi:X \to T \to S$ of $\pi$, 
then $D \not\equiv 0$. 
\end{enumerate}
Then there exists a positive integer $b_0$ such that $bD$ is $\pi$-free 
for any integer $b$ with $b \geq b_0$. 
\end{thm}

\begin{proof}
Replacing $\pi$ by its Stein factorisation, 
we may assume that $\pi_*\MO_X=\MO_S$. 
By perturbing $\Delta$, 
we may assume that $\Delta$ is a $\Q$-divisor such that $K_{X/B}+\Delta$ 
is $\Q$-Cartier, and $D-(K_{X/B}+\Delta)$ is $\pi$-ample.

First we treat the case where $\dim S=0$, i.e. $S=\Spec\,k$ for a field $k$. 
If $k$ is of characteristic zero, 
then the assertion follows from \cite[Theorem 13.1]{Fuj11} 
possibly after replacing $k$ by its algebraic closure. 
If $k$ is of positive characteristic, 
then the assertion follows from the same argument as in \cite[Theorem 3.2]{Tan15} 
after replacing \cite[Theorem 2.6]{Tan15} and \cite[Theorem 2.9]{Tan15} 
by Theorem~\ref{t-glob-klt-kvv} and Theorem~\ref{t-glob-nadel}, respectively. 

Second we consider the remaining case, i.e. $\dim S \geq 1$. 
If $(X, \Delta)$ is klt, 
then we can apply the same argument as in \cite[Theorem 3-1-1]{KMM87} 
after replacing \cite[Theorem 1-2-3]{KMM87} and 
\cite[the Non-Vanishing Theorem]{KMM87} by 
Theorem~\ref{t-rel-klt-kvv} and Theorem~\ref{t-non-vanishing}, respectively. 
Thus we may assume that $(X, \Delta)$ is not klt. 
Since the problem is local on $S$, we may assume that $S$ is affine. 
It suffices to show that $\MO_X(bD)$ is generated by its global sections 
for any large integer $b$, 
which follows from the same argument as in \cite[Theorem 3.2]{Tan15} 
by using Theorem~\ref{t-rel-klt-kvv} and Theorem~\ref{t-rel-nadel} 
instead of \cite[Theorem 2.6]{Tan15} and \cite[Theorem 2.9]{Tan15}, respectively. 
\end{proof}

\begin{rem}
We can not drop the assumption (4) in the statement of Theorem \ref{t-bpf} 
(cf. \cite[Theorem 1.4]{Tanb}). 
\end{rem}

\subsection{Minimal model program  for $\Q$-factorial surfaces}

The purpose of this subsection is to show 
the minimal model program  for $\Q$-factorial surfaces 
(Theorem~\ref{t-qfac-mmp}). 
To this end, we establish a contraction theorem of extremal rays. 

\begin{thm}\label{t-contraction}
Let $B$ be a scheme satisfying Assumption~\ref{a-base}. 
Let $\pi:X \to S$ be a projective $B$-morphism 
from a quasi-projective normal $B$-surface $X$ 
to a quasi-projective $B$-scheme $S$. 
Let $\Delta$ be an $\R$-divisor on $X$ such that $0 \leq \Delta \leq 1$ 
and $K_{X/B}+\Delta$ is $\R$-Cartier. 
Let $R$ be an extremal ray of $\overline{{\rm NE}}(X/S)$ 
which is $(K_{X/B}+\Delta)$-negative. 
Then there exists a projective $S$-morphism $f:X \to Y$ 
to a projective $S$-scheme $Y$ that satisfies the following properties. 
\begin{enumerate}
\item $f_*\MO_X=\MO_Y$. 
\item For any projective $S$-curve $C$ on $X$ such that $\pi(C)$ is one point, 
$f(C)$ is one point if and only if $[C] \in R$. 
\item If $\dim Y \geq 1$, then the sequence 
$$0 \to \Pic\,Y \xrightarrow{f^*} \Pic\,X \xrightarrow{\cdot C} \Z$$
is exact for any projective $S$-curve $C$ on $X$ such that $f(C)$ is one point. 
\item $\rho(Y/S)=\rho(X/S)-1$. 
\item If $X$ is $\Q$-factorial, then $Y$ is $\Q$-factorial and 
the fibre $f^{-1}(y)$ of any closed point $y \in Y$ is irreducible. 
\end{enumerate}
\end{thm}

\begin{proof}
First of all, we will show that there exists a morphism $f:X \to Y$ 
satisfying (1) and (2), and that 
the properties (3)--(5) hold if $\dim Y=0$. 
After that, we will prove that the properties (3), (4) and (5) hold 
if $\dim Y \geq 1$.

Replacing $\pi$ by its Stein factorisation, 
we may assume that $\pi_*\MO_X=\MO_S$. 
By perturbing $\Delta$, 
we may assume that $\Delta$ is a $\Q$-divisor. 
By the cone theorem (Theorem~\ref{t-cone}), we can find 
a $\pi$-nef Cartier divisor $L$ on $X$ 
such that $\overline{{\rm NE}}(X/S) \cap L^{\perp}=R$ and that 
$L-(K_{X/B}+\Delta)$ is $\pi$-ample. 
If $\dim S=0$ (i.e. $S=\Spec\,k$ for some field $k$) and $D \equiv 0$, then 
we can check that the structure morphism $X \to Y:=\Spec\,k$ satisfies 
all the properties (1)--(5). 
Otherwise, by Theorem~\ref{t-bpf}, 
there exists a projective $S$-morphism $f:X \to Y$ 
that satisfies (1) and (2). 
From now on, we assume that $\dim Y \geq 1$.

We show (3). 
Let $M$ be a Cartier divisor on $X$ such that $M \cdot C=0$. 
It suffices to show that there is a Cartier divisor $M_Y$ on $Y$ such that $M=f^*M_Y$. 
We can find 
a $\pi$-nef Cartier divisor $L$ on $X$ 
such that $\overline{{\rm NE}}(X/S) \cap L^{\perp}=R$ and that 
$L-(K_{X/B}+\Delta)$ is $\pi$-ample. 
Replacing $L$ by its multiple, 
we may assume that $L+M$ is $\pi$-nef and $L+M-(K_{X/B}+\Delta)$ is 
$\pi$-ample. 
Since $bL$ and $b(L+M)$ are $\pi$-free for any large integer $b$, 
we have that the divisors $bL, (b+1)L, b(L+M), (b+1)(L+M)$ 
are the pull-backs of Cartier divisors on $Y$. 
Thus also $M$ is the pull-back of a Cartier divisor on $Y$, 
hence (3) holds. 
The assertion (4) directly follows from (3).

We show (5). 
Assume that $X$ is $\Q$-factorial. 
If $\dim Y = 1$, then $Y$ is regular and in particular $\Q$-factorial. 
If $\dim Y = 2$, then we can apply the same argument as in 
\cite[Corollary 3.18]{KM98} by using (3) instead of \cite[Corollary 3.7(4)]{KM98}. 
In any case, $Y$ is $\Q$-factorial. 
We prove that any fibre of $f$ is irreducible. 
We only treat the case where $\dim Y=1$. 
Let $C_1$ and $C_2$ be two curves contained in $f^{-1}(y)$ 
such that $C_1 \neq C_2$ and $C_1 \cap C_2\neq \emptyset$. 
Since $C_1 \cdot C_2>0$ and $C_1^2<0$, 
this contradicts $\rho(X/Y)=1$. 
Thus (5) holds. 
\end{proof}

\begin{thm}\label{t-qfac-mmp}
Theorem~\ref{intro-MMP} holds if 
$(X, \Delta)$ satisfies the condition $({\rm QF})$. 
\end{thm}

\begin{proof}
The assertion follows directly from Theorem~\ref{t-cone} 
and Theorem~\ref{t-contraction}. 
\end{proof}

\subsection{Minimal model program  for log canonical surfaces}

In this subsection, 
we establish the minimal model program  for log canonical surfaces. 
Since we have already got the cone theorem and the contraction theorem, 
it only remains to 
show that numerically log canonical surfaces are log canonical 
(Theorem~\ref{t-numerical-lc}). 
Actually it is necessary to establish this result 
in order to formulate the log canonical minimal model program without the $\Q$-factorial condition. 
Such a result is well-known if the base scheme $B$ is an algebraically closed field (\cite[Proposition 4.11]{KM98}). 
The proof of \cite[Proposition 4.11]{KM98} depends on the classification 
of log canonical surface singularities, whilst our strategy avoid using it. 

\begin{dfn}\label{d-nlc}
Let $B$ be a scheme satisfying Assumption~\ref{a-base}. 
Let $X$ be a quasi-projective normal $B$-surface and 
let $\Delta$ is an $\R$-divisor on $X$. 
We say a pair $(X, \Delta)$ is {\em numerically log canonical} if 
the following properties hold. 
\begin{enumerate}
\item $\Delta$ is effective. 
\item For an arbitrary projective birational morphism 
$f:Y \to X$ and the $f$-exceptional $\R$-divisor $E$ uniquely determined by  
$$K_{Y/B}+f_*^{-1}\Delta+E \equiv_f 0,$$
the inequality $f_*^{-1}\Delta+E \leq 1$ holds. 
Note that $E$ is uniquely defined 
since the intersection matrix of the exceptional locus is negative definite 
(\cite[Theorem 10.1]{Kol13}). 
\end{enumerate}
\end{dfn}

It is obvious that a numerically log canonical pair $(X, \Delta)$ is 
log canonical if and only if $(X, \Delta)$ is a log pair, 
i.e. $K_{X/B}+\Delta$ is $\R$-Cartier.

\begin{thm}\label{t-dlt-blowup}
Let $B$ be a scheme satisfying Assumption~\ref{a-base}. 
Let $X$ be a quasi-projective normal $B$-surface and 
let $\Delta$ be an effective $\R$-divisor on $X$. 
For the irreducible decomposition $\Delta=\sum_{i \in I}d_i D_i$, 
we set $\Delta_1:=\sum_{i \in I}\min\{d_i, 1\}D_i$. 
Then there exists a projective birational $B$-morphism $f:Y \to X$ such that 
\begin{enumerate}
\item 
$(Y, f^{-1}_*\Delta_1+E)$ is a two-dimensional $\Q$-factorial dlt pair over $B$, 
where $E$ is the reduced divisor on $X$ with $\Supp\,E=\Ex(f)$, 
\item $K_{Y/B}+f^{-1}_*\Delta_1+E$ is $f$-nef, and 
\item $K_{Y/B}+f^{-1}_*\Delta+E+E' \equiv_f 0$  
for some $f$-exceptional effective $\R$-divisor $E'$. 
\end{enumerate}
\end{thm}

A projective birational morphism satisfying the above properties (1)--(3) 
is called a dlt blowup of $(X, \Delta)$.

\begin{proof}
Let $\Delta_2:=\Delta-\Delta_1$. 
We take a log resolution $\mu:W \to X$ of $(X, \Delta)$. 
We can write 
$$K_{W/B}+\mu^{-1}_*\Delta_1+E_W \equiv_{\mu} -\mu^{-1}_*\Delta_2+E'_W$$
where $E_W$ is the $\mu$-exceptional reduced divisor with $\Supp E_W=\Ex(\mu)$ and $E'_W$ is an $\mu$-exceptional $\R$-divisor. 
We run a $(K_{W/B}+\mu^{-1}_*\Delta_1+E_W)$-MMP over $X$ (Theorem~\ref{t-qfac-mmp}), 
and set $f:Y \to X$ to be the end result. 
Let $\mu_Y:W \to Y$ be the induced morphism, 
$E:=(\mu_Y)_*E_W$ and $E':=-(\mu_Y)_*E_W'$. 
Then (1) and (2) hold.  
By the negativity lemma (Lemma~\ref{l-negativity}), we get 
$$K_{W/B}+\mu^{-1}_*\Delta+E_W-E'_W=\mu_Y^*(K_{Y/B}+f^{-1}_*\Delta+E+E'),$$
which implies $K_{Y/B}+f^{-1}_*\Delta_1+E \equiv_f -f^{-1}_*\Delta_2-E'$. 
It follows again from the negativity lemma (Lemma~\ref{l-negativity}) that $E'$ is effective, 
hence (3) holds. 
\end{proof}

\begin{rem}\label{r-dlt-blowup}
We use the same notation as in Theorem~\ref{t-dlt-blowup}. 
\begin{enumerate}
\item[(a)] If $(X, \Delta)$ is a log pair i.e. $K_{X/B}+\Delta$ is $\R$-Cartier, 
then Theorem~\ref{t-dlt-blowup}(3) implies that $K_{Y/B}+f_*^{-1}\Delta+E+E'=f^*(K_{X/B}+\Delta)$. 
\item[(b)] $(X, \Delta)$ is numerically log canonical 
if and only if $\Delta=\Delta_1$ and $E'$ in Theorem~\ref{t-dlt-blowup}(3) is zero. 
\end{enumerate}
\end{rem}

\begin{lem}\label{lcgermsing}
Let $B$ be a scheme satisfying Assumption~\ref{a-base}. 
Let $(X, \Delta)$ be a numerically log canonical quasi-projective $B$-surface. 
Assume that $x\in X$ is a unique non-regular closed point of $X$ and that $(X\setminus\{x\}, \Delta|_{X \setminus \{x\}})$ is dlt. 
Then one of the following assertions hold. 
\begin{enumerate}
\item{$X$ is $\Q$-factorial.}
\item{$x \not \in \Supp \Delta$ and there exists a projective birational $B$-morphism 
$$g:Z \to X$$
from a normal $\Q$-factorial quasi-projective $B$-surface $Z$ 
such that $E:=\Ex(g)$ is a prime divisor and that $(K_{Z/B}+E)\cdot E=0$. 
Moreover there exists a dlt blowup $f:Y \to X$ of $(X, \Delta)$ 
(cf. Theorem~\ref{t-dlt-blowup}) that factors through $g$. }
\end{enumerate}
\end{lem}

\begin{proof}
Let $f:Y \to X$ be a dlt blowup whose existence is guaranteed 
in Theorem~\ref{t-dlt-blowup}. 
Since $(X\setminus\{x\}, \Delta|_{X \setminus \{x\}})$ is dlt, 
possibly after replacing $Y$ by the scheme obtained by gluing $X \setminus \{x\}$ 
and $f^{-1}(X^0)$ for sufficiently small open neighbourhood $X^0$ of $x \in X$, 
we may assume that $f(\Ex(f))=\{x\}$. 
For the sum $\sum_{i=1}^r E_i$ of the $f$-exceptional prime divisors $E_i$, 
we can write as follows (cf. Remark~\ref{r-dlt-blowup}):
$$K_{Y/B}+f_*^{-1}\Delta+\sum_{i=1}^rE_i \equiv_f 0.$$

\begin{claim}
There exists a projective birational $B$-morphism $g:Z \to X$ 
of quasi-projective normal $B$-surfaces which satisfies the following properties: 
\begin{enumerate}
\item[(i)] $Z$ is $\Q$-factorial, 
\item[(ii)] $E:=\Ex(g)$ is a prime divisor, 
\item[(iii)] $(K_{Z/B}+g_*^{-1}\Delta+E) \cdot E=0$, and 
\item[(iv)] $f$ factors through $g$. 
\end{enumerate}
\end{claim}

\begin{proof}[Proof of Claim]
We run a $(K_{Y/B}+f_*^{-1}\Delta+\sum_{i=2}^rE_i)$-MMP over $X$ 
(Theorem~\ref{t-qfac-mmp}). 
We show that the end result $g:Z \to X$ satisfies (i)--(iv). 
Let $h:Y \to Z$ be the induced birational morphism. 
It is clear that (i) and (iv) hold. 
Since $K_{Y/B}+g_*^{-1}\Delta+\sum_{i=2}^rE_i \equiv_f -E_1$, 
each step of this MMP contracts $E_i$ for some $i \in \{2, \cdots, r\}$. 
In particular, we get $h_*E_1 \neq 0$. 
Furthermore, all of $E_2, \cdots, E_r$ are contracted, 
since $h_*E_1\neq 0$, $-h_*E_1$ is $g$-nef and $g^{-1}(x)$ is connected. 
Therefore, (ii) holds and $E=h_*E_1$. 
It follows from $K_{Y/B}+f_*^{-1}\Delta+\sum_{i=1}^rE_i \equiv_f 0$ that 
$$K_{Y/B}+f_*^{-1}\Delta+\sum_{i=1}^rE_i=h^*(K_{Z/B}+g_*^{-1}\Delta+E).$$
Thus (iii) holds. 
\end{proof}

Assuming that $x \in \Supp \Delta$, 
it suffices to show that $X$ is $\Q$-factorial. 
By (i)--(iii), the projective birational morphism $g:Z \to X$ 
is the contraction of a $(K_{Z/B}+\frac{1}{2}g_*^{-1}\Delta+E)$-negative 
extremal ray of $\overline{\text{NE}}(Z/X)$. 
It follows from Theorem~\ref{t-contraction} that $X$ is $\Q$-factorial, as desired. 
It completes the proof of Lemma~\ref{lcgermsing}. 
\end{proof}

\begin{rem}\label{r-ZE-lc}
Since a pair $(Z, E)$, appearing in Lemma~\ref{lcgermsing}(2), 
is numerically log canonical and $Z$ is $\Q$-factorial, 
$(Z, E)$ is log canonical. 
\end{rem}

\begin{cor}\label{c-dlt-Qfac}
Let $B$ be a scheme satisfying Assumption~\ref{a-base}. 
Let $(X, \Delta)$ be a two-dimensional dlt pair over $B$. 
Then $X$ is $\Q$-factorial. 
\end{cor}

\begin{proof}
The assertion directly follows from Lemma~\ref{lcgermsing}. 
\end{proof}

\begin{lem}\label{l-slc}
Let $(C, D)$ be a one-dimensional projective semi log canonical pair 
over a field $k$ in the sense of \cite[Definition-Lemma 5.10]{Kol13}. 
If $C$ is irreducible and $K_C+D \equiv 0$, 
then there exists a positive integer $m$ such that 
$\MO_C(m(K_C+D)) \simeq \MO_C$. 
\end{lem}

\begin{proof}
Since $C$ is semi log canonical, the singularities of $C$ is at worst nodal. 
In particular $C$ is Gorenstein. 
If $D \neq 0$, then $\omega_C^{-1}$ is ample and the assertion follows from 
\cite[Lemma 10.6]{Kol13}. 
Thus we may assume that $D=0$. 
Then we get $\omega_C \simeq \MO_C$ by 
$$h^0(X, \omega_X)=h^1(X, \omega_X)=h^0(X, \MO_X) \neq 0,$$
where the first (resp. second) equation 
follows from Corollary~\ref{genus-formula} (resp. Serre duality).
\end{proof}

\begin{thm}\label{t-numerical-lc}
Let $B$ be a scheme satisfying Assumption~\ref{a-base}. 
Let $(X, \Delta)$ be a numerically log canonical quasi-projective $B$-surface. 
Then the following assertions hold. 
\begin{enumerate}
\item{$K_{X/B}$ and all the irreducible components of $\Delta$ are $\Q$-Cartier. 
In particular, $(X, \Delta)$ is log canonical, that is, $K_{X/B}+\Delta$ is $\R$-Cartier.}
\item{There exists an open subset $X^0 \subset X$ 
such that $\Supp \Delta \subset X^0$ and that $X^0$ is $\Q$-factorial. }
\end{enumerate}
\end{thm}

\begin{proof}
By Lemma~\ref{lcgermsing}, 
we may assume that $\Delta=0$ and $x \in X$ is a unique non-regular point. 
It suffices to show that $K_{X/B}$ is $\Q$-Cartier. 
Let 
$$f:Y \xrightarrow{h} Z \xrightarrow{g} X$$ 
be birational morphisms as in Lemma~\ref{lcgermsing}(2). 
We have that 
$$K_{Y/B}+E_Y=h^*(K_{Z/B}+E) \equiv_f 0$$
where $E_Y$ is the reduced divisor with $\Supp E_Y=\Ex(f)$.

\begin{claim}
There exists a positive integer $m$ such that 
$\MO_Z(m(K_{Z/B}+E))|_E \simeq \MO_E$. 
\end{claim}

We now show Claim. 
It follows from \cite[Theorem 2.31]{Kol13} 
that $E$ is nodal and hence Gorenstein. 
Let $(E, \text{Diff}_E(0))$ be a pair defined in \cite[Definition 4.2]{Kol13}. 
Since $E$ is nodal and $(Z, E)$ is log canonical (Remark~\ref{r-ZE-lc}), 
it follows from \cite[Lemma 4.8]{Kol13} that 
$(E, \text{Diff}_E(0))$ is semi log canonical in the sense of \cite[Definition-Lemma 5.10]{Kol13}. 
Then Claim holds by $(K_{Z/B}+E) \cdot E=0$ and Lemma~\ref{l-slc}. 

\medskip

Since there is a natural morphism $E_Y \to E$, 
Claim implies that $\MO_Y(m(K_{Y/B}+E_Y))|_{E_Y} \simeq \MO_{E_Y}$ 
for sufficiently divisible $m \in \Z_{>0}$. 
The exact sequence: 
$$0 \to \MO_Y(-E_Y) \to \MO_Y \to \MO_{E_Y} \to 0,$$
induces an exact sequence:  
$$f_*\MO_Y(m(K_{Y/B}+E_Y)) \to f_*(\MO_Y(m(K_{Y/B}+E_Y))|_{E_Y})$$ 
$$\to R^1f_*\MO_Y(m(K_{Y/B}+E_Y)-E_Y)=0$$
where the last equation follows from Theorem~\ref{t-rel-klt-kvv}. 
Thus, $K_{Y/B}+E_Y$ is $f$-semi-ample, which implies that 
there is a $\Q$-Cartier $\Q$-divisor $L$ on $X$ such that 
$K_{Y/B}+E_Y=f^*L$. 
Taking the push-forward $f_*$, it follows that $K_{X/B}$ is $\Q$-Cartier. 
\end{proof}

\begin{rem}\label{r-known-nlc}
If $B$ is essentially of finite type over a field of characteristic zero 
then Theorem~\ref{t-numerical-lc} was known to follow from 
\cite[Corollary 1.6]{HX}. 
\end{rem}

\begin{thm}\label{t-lc-mmp}
Theorem~\ref{intro-MMP} holds if 
$(X, \Delta)$ satisfies the condition $({\rm LC})$. 
\end{thm}

\begin{proof}
The assertion follows directly from Theorem~\ref{t-cone}, 
Theorem~\ref{t-contraction} and Theorem~\ref{t-numerical-lc}. 
\end{proof}

\section{Miscellaneous results}

In this section, we prove inversion of adjunction (Theorem~\ref{t-IOA}) 
and the Koll\'ar--Shokurov 
connectedness theorem (Theorem \ref{t-connected}) 
for surfaces over excellent schemes. 

\subsection{Inversion of adjunction}

\begin{thm}\label{t-IOA}
Let $B$ be a scheme satisfying Assumption~\ref{a-base}. 
Let $(X, C+D)$ be 
a two-dimensional quasi-projective log pair over $B$, 
where $C$ is a reduced divisor that has no common components 
with an effective $\R$-divisor $D$. 
Let $C^N$ be the normalisation of $C$ and 
let $D_{C^N}:={\rm Diff}_{C^N}(D)$ $($cf. \cite[Definition 2.34]{Kol13}$)$. 
Then the following hold. 
\begin{enumerate}
\item 
$(X, C+D)$ is log canonical around $C$ if and only if 
$(C^N, D_{C^N})$ is log canonical. 
\item 
$(X, C+D)$ is plt around $C$ if and only if 
$(C^N, D_{C^N})$ is klt. 
\end{enumerate}
\end{thm}

\begin{proof}
For both of (1) and (2), 
we only show the if-part since 
the opposite implications follow from \cite[Lemma 4.8]{Kol13}.

Since the problem is local on $X$, 
we fix a closed point $x \in C \subset X$ around which we work. 
In particular, we may assume that 
all the irreducible components of $C+D$ contain $x$. 
Let $f:Y \to X$ be a dlt blowup of $(X, \Delta:=C+D)$ 
as in Theorem~\ref{t-dlt-blowup} and we use the same notation as there. 
Let $\Delta_2:=\Delta-\Delta_1$. 
We have that 
$$K_{Y/B}+f^{-1}_*\Delta_1+E=f^*(K_{X/B}+C+D)-(f^{-1}_*\Delta_2+E')$$
is $f$-nef (cf. Remark~\ref{r-dlt-blowup}). 

We show (1). 
We may assume that $f(\Ex(f))=\{x\}$. 
Assume that $(X, C+D)$ is not log canonical around $x$, 
which is equivalent to $f^{-1}_*\Delta_2+E' \neq 0$ (Remark~\ref{r-dlt-blowup}). 
Since $-(f^{-1}_*\Delta_2+E')$ is $f$-nef, 
it follows from the negativity lemma (Lemma~\ref{l-negativity}) 
that $\Supp\,E'$ contains $f^{-1}(x)$. 
This implies that $(C^N, D_{C^N})$ is not log canonical by \cite[Proposition 4.5(2)]{Kol13}, 
hence (1) holds. 

We show (2). 
Assume that $(C^N, D_{C^N})$ is klt. 
It suffices to show that $(X, C+D)$ is plt. 
It follows from (1) that $(X, C+D)$ is log canonical, 
hence  $\Delta=\Delta_1$, $\Delta_2=0$ and $E'=0$. 

We now prove that $f$ is an isomorphism. 
We assume by contradiction that $E \neq 0$. 
Since we work around $x$, we may assume that $f(\Ex(f))=\{x\}$. 
Since Theorem~\ref{t-dlt-blowup}(1) implies $\Supp E=\Ex(f)$, 
we have that $f_*^{-1}C$ intersects $E$. 
Since $E \neq 0$, we get $\llcorner D_{C^N}\lrcorner \neq 0$ by \cite[Proposition 4.5(6)]{Kol13}, which contradicts the fact that 
$(C^N, D_{C^N})$ is klt. 
Thus $f$ is an isomorphism, as desired. 

Since $f$ is an isomorphism, $(X, C+D)$ is dlt. 
If $\llcorner D\lrcorner \neq 0$, then $(C^N, D_{C^N})$ is not klt 
again by \cite[Proposition 4.5(6)]{Kol13}. 
Thus we get $\llcorner D\lrcorner=0$, hence (2) holds. 
\end{proof}

\subsection{Connectedness theorem}

\begin{thm}\label{t-connected}
Let $B$ be a scheme satisfying Assumption~\ref{a-base}. 
Let $\pi:X \to S$ be a projective $B$-morphism 
from a two-dimensional quasi-projective log pair $(X, \Delta)$ over $B$ 
to a quasi-projective $B$-scheme $S$ with $\pi_*\MO_X=\MO_S$. 
Let $\Nklt(X, \Delta)$ be the reduced closed subscheme of $X$ 
consisting of the non-klt points of $(X, \Delta)$. 
If $-(K_{X/B}+\Delta)$ is $\pi$-nef and $\pi$-big, then 
any fibre of the induced morphism $\Nklt(X, \Delta) \to S$ is 
either empty or geometrically connected. 
\end{thm}

\begin{proof}
We divide the proof into three steps. 

\begin{step}\label{step-ample}
There exists an effective $\R$-divisor $\Delta'$ such that 
$(X, \Delta')$ is a log pair over $B$, $\Delta' \geq \Delta$, 
$\Nklt(X, \Delta)=\Nklt(X, \Delta')$, and 
$-(K_{X/B}+\Delta')$ is $\pi$-ample. 
\end{step}

\begin{proof}[Proof of Step~\ref{step-ample}]
Since $N:=-(K_{X/B}+\Delta)$ is $\pi$-big, we can write 
$$N=A+D$$
for some $\pi$-ample $\R$-Cartier $\R$-divisor $A$ and 
an effective $\R$-divisor $D$. 
Thus for any rational number $0<\epsilon<1$, 
we have that 
$$-(K_{X/B}+\Delta)=N=(1-\epsilon)N+\epsilon A+\epsilon D,$$
where $(1-\epsilon)N+\epsilon A$ is $\pi$-ample. 
Thus it suffices to find a rational number $0<\epsilon<1$ such that 
$\Nklt(X, \Delta)=\Nklt(X, \Delta+\epsilon D)$. 
We can find such a number $\epsilon$ by taking a log resolution 
of $(X, \Delta+D)$. 
\end{proof}

\begin{step}\label{step-relative}
If $\dim S \geq 1$, then the assertion in the theorem holds. 
\end{step}

\begin{proof}[Proof of Step~\ref{step-relative}]
By Step~\ref{step-ample}, 
we may assume that $-(K_{X/B}+\Delta)$ is $\pi$-ample. 
We have an exact sequence 
$$0 \to \mathcal J_{\Delta} \to \MO_X \to \MO_{W} \to 0$$
where $W$ is the closed subscheme corresponding to the multiplier ideal $\mathcal J_{\Delta}$. 
Since the support of $W$ is equal to the non-klt locus of $(X, \Delta)$, 
we get $W_{\red}=\Nklt(X, \Delta)$. 
Since $-(K_{X/B}+\Delta)$ is $\pi$-ample, 
it follows from Theorem~\ref{t-rel-nadel} that $R^1\pi_*\mathcal J_{\Delta}=0$, 
which implies that the induced homomorphism 
$$\rho:\MO_S=\pi_*\MO_X \to \pi_*\MO_{W}$$
is surjective. 
Since $\rho$ factors through $\MO_{\pi(W)}$, 
we have that 
$$\MO_{\pi(W)} \to \MO_V=\pi_*\MO_{W}$$ 
is surjective, 
where $W \to V \to \pi(W)$ is the Stein factorisation of $W \to \pi(W)$. 
Thus $V \to \pi(W)$ is a surjective closed immersion, hence a universal homeomorphism. 
In particular, any fibre of $W \to \pi(W)$ is geometrically connected. 
Since a surjective closed immersion $\Nklt(X, \Delta) \to W$ 
is a universal homeomorphism, 
it completes the proof of Step~\ref{step-relative}. 
\end{proof}

\begin{step}\label{step-absolute}
If $\dim S=0$, then the assertion in the theorem holds. 
\end{step}

\begin{proof}[Proof of Step~\ref{step-absolute}]
Since $\dim S=0$, we have that $S=\Spec\,k$ for a field $k$. 
Taking the base change to the separable closure of $k$, 
we may assume that $k$ is separably closed. 
In particular, 
it suffices to show that $\Nklt(X, \Delta)$ is connected. 
By replacing $(X, \Delta)$ by its dlt blowup as in Theorem~\ref{t-dlt-blowup}, 
we may assume that $(X, \Delta^{<1})$ is klt and 
$\Nklt(X, \Delta)=\Supp\,\Delta^{\geq 1}$ (cf. Remark~\ref{r-dlt-blowup}). 
Set $A:=-(K_X+\Delta)$. 
By Step~\ref{step-ample}, we may assume that $A$ is ample. 
We run a $(K_X+\Delta^{<1}+A)$-MMP: $f:X \to Y$. 
Since $-(K_X+\Delta^{<1}+A)=\Delta^{\geq 1} \neq 0$, 
the program ends with a Mori fibre space $\rho:Y \to T$. 
We set $\Delta_Y:=f_*\Delta$. 
We now prove:

\begin{claim}
The following hold. 
\begin{enumerate}
\item $f(\Supp\,\Delta^{\geq 1})=\Supp\,\Delta_Y^{\geq 1}$. 
\item If $\Nklt(Y, \Delta_Y)$ is connected, then $\Nklt(X, \Delta)$ is connected. 
\item $\Nklt(Y, \Delta_Y)=\Supp\,\Delta_Y^{\geq 1}$. 
\end{enumerate}
\end{claim}

Since $f$ is the composition of the extremal contractions, 
we only prove Claim under the assumption that $E:=\Ex(f)$ is irreducible. 
In particular, $\Delta^{\geq 1}$ is $f$-ample. 

We show (1). 
The inclusion $f(\Supp\,\Delta^{\geq 1})\supset \Supp\,\Delta_Y^{\geq 1}$ 
is clear. 
It suffices to show that $f(E) \in \Supp\,\Delta_Y^{\geq 1}$. 
Since $\Delta^{\geq 1}$ is $f$-ample, 
there is an irreducible component $C$ of $\Supp\,\Delta^{\geq 1}$ 
such that $C \cdot E>0$. 
It follows from $E^2<0$ that $C\neq E$. 
Thus, we get $f(E) \in \Supp\,\Delta_Y^{\geq 1}$, 
hence (1) holds.

We show (2). 
We assume by contradiction that $\Supp\,\Delta_Y^{\geq 1}$ is connected 
but $\Supp\,\Delta^{\geq 1}$ is not. 
Thanks to (1), we can find a point $y \in \Supp\,\Delta_Y^{\geq 1}$ such that 
$f^{-1}(y) \cap \Supp\,\Delta^{\geq 1}$ is not connected. 
This contradicts Step~\ref{step-relative}. 
Thus (2) holds.

The assertion (3) follows from (1) and the fact that the MMP $f:X \to Y$ 
can be considered as a $(K_X+\Delta)$-MMP, because 
the ampleness of $A$ is preserved under the MMP. 
It completes the proof of Claim.

\medskip

If $\dim T=0$, then we have $\rho(Y)=1$, 
hence $\Nklt(Y, \Delta_Y)=\Supp\,\Delta_Y^{\geq 1}$ is clearly connected. 
Thus we may assume that $\dim T=1$. 
Assuming that $\Nklt(Y, \Delta_Y)$ is not connected, let us derive a contradiction. 
Let $D_1$ and $D_2$ be distinct connected components of $\Nklt(Y, \Delta_Y)$. 
Since $\Delta_Y^{\geq 1}$ is $\rho$-ample, 
$\Nklt(Y, \Delta_Y)=\Supp\,\Delta_Y^{\geq 1}$ dominates $T$. 
In particular, we may assume that $D_1$ dominates $T$. 
On the other hand, Step \ref{step-relative} implies that 
$\Nklt(Y, \Delta_Y) \cap \rho^{-1}(t)$ is connected for any closed point $t \in T$. 
In particular, $D_2$ does not dominate $T$. 
Since $D_2$ is connected, we have that $D_2 \subset \rho^{-1}(t_0)$ 
for some closed point $t_0 \in T$. 
However, this implies that $\Nklt(Y, \Delta_Y) \cap \rho^{-1}(t_0)$ has at least 
two connected components: $D_2$ and 
a connected component of $D_1 \cap \rho^{-1}(t_0)$. 
This is a contradiction. 
It completes the proof of  Step~\ref{step-absolute}.
\end{proof}
Thus the assertion in Theorem holds by Step~\ref{step-relative} and Step~\ref{step-absolute}. 
\end{proof}

\end{document}